\documentclass[a4paper, 12pt]{article}

\usepackage{amsfonts, amsmath, amsthm}
\usepackage{amssymb}
\usepackage{latexsym}
\usepackage[dvips]{graphicx}
\usepackage{color}

\theoremstyle{plain}
\newtheorem{thm}{Theorem}[section]

\newtheorem{lem}{Lemma}[section]

\theoremstyle{definition}
\newtheorem{dfn}{Definition}[section]

\theoremstyle{definition}
\newtheorem{rmk}{Remark}[section]

\numberwithin{equation}{section}

\newcommand{\R}{\mathbb{R}}
\newcommand{\C}{\mathbb{C}}

\newcommand{\cc}[1]{\overline{#1}}
\newcommand{\op}[1]{\mathcal{#1}}

\newcommand{\pa}{\partial}
\newcommand{\eps}{\varepsilon}
\newcommand{\jb}[1]{\langle #1 \rangle}

\newcommand{\realpart}{{\rm Re\,}}
\newcommand{\imagpart}{{\rm Im\,}}

\newcommand{\dis}{\displaystyle}

\begin{document}
\title{
Upper and lower $L^2$-decay bounds for a class of 
derivative nonlinear Schr\"odinger equations\\
 }

\author{
             Chunhua Li  \thanks{
              Department of Mathematics, College of Science,
              Yanbian University.
              977 Gongyuan Road, Yanji, Jilin 133002, China.
              (E-mail: {\tt sxlch@ybu.edu.cn})
             }
          \and
          Yoshinori Nishii\thanks{
             Department of Mathematics, 
              Tokyo University of Science. 
              1-3 Kagurazaka, Shinjuku-ku, Tokyo 162-8601, Japan.
              (E-mail: {\tt yoshinori.nishii@rs.tus.ac.jp}) 
             }
          \and
          Yuji Sagawa \thanks{
              Department of Mathematics, 
              Chiba Institute of Technology.
              2-1-1 Shibazono, Narashino, Chiba 275-0023, Japan. 
              (E-mail: {\tt sagawa.yuji@p.chibakoudai.jp})
             }
          \and
          Hideaki Sunagawa \thanks{
              Department of Mathematics, Graduate School of Science, 
              Osaka Metropolitan  University.
              3-3-138 Sugimoto, Sumiyoshi-ku, Osaka 558-8585, Japan. 
              (E-mail: {\tt sunagawa@omu.ac.jp})
             }
}

\date{\today }
\maketitle

\noindent{\bf Abstract:}\ 
We consider the initial value problem for cubic derivative 
nonlinear Schr\"odinger equations possessing weakly dissipative structure
in one space dimension. 
We show that the small data solution decays like $O((\log t)^{-1/4})$ 
in $L^2$ as $t\to +\infty$. Furthermore, we find that this $L^2$-decay rate 
is optimal by giving a lower estimate of the same order. 
\\

\\



\section{Introduction and the main results}  \label{sec_intro}
This article is devoted to the study on large-time behavior of solutions to 
the initial value problem for one-dimensional derivative nonlinear Schr\"odinger equation 
in the form 
\begin{align}
i\pa_t u + \tfrac{1}{2}\pa_x^2 u=N(u,\pa_x u),
\qquad  t>0,\ x\in \R
\label{eq}
\end{align}
with 
\begin{align}
u(0,x)=\varphi(x), 
\qquad  x\in \R,
\label{data}
\end{align}where $i=\sqrt{-1}$ $\pa_t=\pa/\pa t$, $\pa_x=\pa/\pa x$, and 
$\varphi$ 
is a prescribed $\C$-valued function on $\R$. 
The nonlinear term $N(u,\pa_x u)$ is a cubic homogeneous polynomial in  
$(u,\cc{u}, \pa_x u, \cc{\pa_x u})$ with complex coefficients (which we refer to 
a ``cubic nonlinear term'' throughout this article). 
Our goal is to refine the previous result \cite{LNSS3}  
from the viewpoint of $L^2$-decay rate  under 
a certain structural condition on $N$ mentioned below. 

Before getting into the details, let us recall why and how we are interested 
in this problem.
As is well-known, cubic nonlinearity gives a borderline between short-range 
and long-range situation for the one-dimensional nonlinear Schr\"odinger 
equations. 
To see typical nonlinear effects, let us first focus on the power-type 
nonlinearity case 
\begin{align}
 i\pa_t u + \tfrac{1}{2}\pa_x^2 u = \lambda |u|^2 u.
 \label{nls_1}
\end{align}
According to Hayashi--Naumkin \cite{HN1}, the small data solution $u(t,x)$ 
to  \eqref{nls_1} with $\lambda \in \R\backslash\{0\}$ behaves like 
\[
 u(t,x)=\frac{1}{\sqrt{it}} \alpha^{\pm} (x/t) 
 e^{i\{\frac{x^2}{2t}  -\lambda  |\alpha^{\pm}(x/t)|^2 \log t \}}
 +o(t^{-1/2})
\quad \mbox{in} \ \ L^{\infty}(\R_x)
\]
as $t\to \pm \infty$, where  $\alpha^{\pm}$ is a suitable 
$\C$-valued function on $\R$. 
The additional logarithmic correction in the phase reflects a typical 
long-range character. 
If $\lambda\in \C$ in \eqref{nls_1}, another kind of long-range effect can 
be observed.
Indeed, it is shown in  \cite{Shim} that the small data solution $u(t,x)$ 
to \eqref{nls_1} decays like 
$O(t^{-1/2}(\log t)^{-1/2})$ in $L^{\infty}(\R_x)$ as $t\to +\infty$ 
if $\imagpart\lambda<0$. 
This gain of additional logarithmic time decay should be interpreted as 
another kind of long-range effect 
(see also \cite{CH}, \cite{CHN}, \cite{HLN}, \cite{HLN2}, \cite{Hoshino}, 
\cite{Hoshino2}, \cite{JJL}, \cite{KLS}, \cite{Kim}, \cite{KitaLi}, 
\cite{KitaNak}, \cite{KitaShim}, \cite{LNSS1}, \cite{LNSS2}, \cite{OgSat}, 
\cite{Sat1}, \cite{Sat2}, and so on). 
Time decay in $L^2$-norm is also investigated by several authors. 
Among others, it is pointed out by Kita-Sato \cite{KitaSato} that 
 optimal $L^2$-decay rate is $O((\log t)^{-1/2})$ in the case of 
\eqref{nls_1} with $\imagpart\lambda<0$ (see also Section~\ref{sec_strict} 
below for the related issue). 

Next we turn our attentions to more general derivative nonlinearity case 
\eqref{eq}. In what follows, we make an assumption that
\begin{align}
 N(e^{i\theta},0) =e^{i \theta} N(1,0), \qquad \theta \in \R,
\label{weak_gi}
\end{align}
to avoid the worst terms $u^3$, $\cc{u}^3$, $u\cc{u}^2$ 
(it is known that these three terms are quite difficult to handle in the 
present setting, and we do not pursue this case here). 
As pointed out in \cite{SagSu} (see also \cite{Su2}, \cite{Su3}), 
typical results on  large-time  behavior 
of global solutions to \eqref{eq} under \eqref{weak_gi} can be summarized in 
terms of the function $\nu:\R\to \C$ defined by 
\begin{align*}
 \nu(\xi)=\frac{1}{2\pi i} \oint_{|z|=1} N(z,i\xi z) \frac{dz}{z^2}
\end{align*}
(see \ref{sec_app_a} for some backgrounds on $\nu(\xi)$).
For instance, 
\begin{itemize}
\item
Small data global existence holds in $C([0,\infty);H^3\cap H^{2,1})$ 
under the condition
\begin{align}
\imagpart \nu(\xi) \le 0,\quad \xi \in \R,
\tag{${\bf A}$}
\end{align}
where $H^s$ denotes the $L^2$-based Sobolev space of order $s$, 
and the weighted Sobolev space $H^{k,m}$ is defined by 
$\{\phi \in L^2\,|\, \jb{\, \cdot \,}^{m} \phi \in H^k\}$ with 
$\jb{x}=\sqrt{1+x^2}$. 
See e.g., \cite{HN2}, \cite{HNS}, etc., for details on global existence.
\item
The global solution has (at most) logarithmic phase correction if 
\begin{align}
\imagpart \nu(\xi) = 0,\quad \xi \in \R.
\tag{${\bf A}_0$}
\end{align}
Also it is not difficult to see that  there is no $L^2$-decay under 
(${\bf A}_0$) for generic initial data of small amplitude.
\item
$L^2$-decay of the global solution does occur under the condition 
\begin{align}
\sup_{\xi \in \R} \imagpart \nu(\xi)<0.
\tag{${\bf A}_+$}
\end{align}
In \cite{LiS} and \cite{LNSS3}, an $L^2$-decay estimate of order  
$O((\log t)^{-3/8+\delta})$ with $\delta>0$ is derived, 
but it seems far from optimal in view of several recent results 
(such as \cite{HLN}, \cite{HLN2}, \cite{Sat1}, \cite{KitaSato}, etc.) 
on the power-type nonlinearity case. 
See Remark~\ref{rmk_upper_str} below for more details. 
\item
Under the stronger condition 
\begin{align}
\imagpart \nu(\xi) \le -C_*\jb{\xi}^2, \quad \xi \in \R 
\tag{${\bf A}_{++}$}
\end{align}
with some $C_*>0$, we have $L^2$-decay of order $O((\log t)^{-1/2})$, 
according to \cite{LiS}. 

\end{itemize}

In the previous paper \cite{LNSS3}, we have pointed out that an interesting 
case is not covered by these classifications, that is the case where  
$({\bf A})$ is safistied but $({\bf A}_0)$ and $({\bf A}_+)$ are violated 
(for example, if $N=-i|u_x|^2u$, we can easily check that 
$\imagpart \nu(\xi)=-\xi^2\le 0$, 
while the inequality is not strict because of vanishing at $\xi=0$). 
This is what we are going to address here.

To going further, let us recall the following lemma given in \cite{LNSS3}.
\begin{lem} \label{lem_weak_dissip}
Let $N$ be a cubic nonlinear term satisfying 
\eqref{weak_gi}.
If $({\bf A})$ is safistied but $({\bf A}_0)$ and $({\bf A}_+)$ are violated, 
then there exist $c_0>0$ and $\xi_0\in \R$ such that 
$\imagpart \nu(\xi)=-c_0(\xi-\xi_0)^2$. 
The converse is also true.
\end{lem}
This lemma naturally leads us to the following definition. 

\begin{dfn} \label{dfn_wd}
We say that a cubic nonlinear term $N$  
is {\em{weakly dissipative}} if  the following two conditions are satisfied: 
\begin{itemize}
\item[(i)] $N(e^{i\theta},0) =e^{i \theta} N(1,0)$ for $\theta \in \R$.
\item[(ii)]  There exist $c_0>0$ and $\xi_0\in \R$ such that 
$\imagpart \nu(\xi)=-c_0(\xi-\xi_0)^2$.
\end{itemize}
\end{dfn}

In \cite{LNSS3}, it has been shown  that
\begin{align}
\|u(t)\|_{L^2}\le \frac{C\eps}{(1+ \eps^2\log (t+1))^{1/4-\delta}}
\label{est_LNSS3}
\end{align}
for $t\geq 0$ with an arbitrarily small $\delta>0$, provided that $N$ is 
weakly dissipative and $\eps =\|\varphi\|_{H^3\cap H^{2,1}}$ is sufficiently 
small. 
However, it is not clear whether this estimate is sharp or not. Indeed, the 
proof in \cite{LNSS3} heavily relies on the pointwise estimate for 
$\alpha(t,\xi)=\op{F}[\op{U}(-t)u(t,\cdot)](\xi)$ in the form 
\begin{align*}
 |\alpha(t,\xi)| 
\le 
\frac {C\eps}{(\eps^2 \log t)^{\theta/2}}
\frac {1}{|\xi-\xi_0|^{\theta}\jb{\xi}^{2-2\theta}}, 
\quad \theta\in [0,1]
\end{align*}
for $t\geq 2$, where $\op{F}$ denotes the Fourier transform and $\op{U}(t)$ 
is the free Schr\"odinger propagater 
(see Section~\ref{subsec_reduction} below). 
This estimate allows us to reduce the problem to finding  
the admissible range of the parameter $\theta$ for convergence of the integral 
\begin{align*}
\int_{\R} \frac{d\xi}{|\xi-\xi_0|^{2\theta} \jb{\xi}^{4-4\theta}}.
\end{align*}
A nuisance exponent $\delta>0$ comes from the choice of $\theta=1/2-2\delta$. 

The aim of this article is  to remove 
an extra $\delta>0$ from \eqref{est_LNSS3} by an alternative approach. 
We will also show that the $L^2$ decay rate $O((\log t)^{-1/4})$ of global 
solutions to \eqref{eq} is optimal in the weakly dissipative case.

Our first result is as follows.
\begin{thm} \label{thm_upper}
Suppose that  $N$ is weakly dissipative 
and that $\eps =\|\varphi\|_{H^3\cap H^{2,1}}$ is sufficiently small. 
Then there exists a positive constant $C$, not depending on $\eps$, 
such that the global solution $u$ to \eqref{eq}--\eqref{data} satisfies
\[
\|u(t)\|_{L_x^2}\le \frac{C\eps}{(1+ \eps^2\log (t+1))^{1/4}}
\] 
for $t\ge 0$.
\end{thm}

To state our second result, we put a small parameter $\eps$ in front of
the initial data explicitly so that the information on the amplitude 
is separated from the others,
that is, we replace the initial condition \eqref{data} by
\begin{align}
 u(0,x) =\eps \psi(x),\qquad x\in \R,
\label{data_eps}
\end{align}
where $\psi\in H^{3}\cap H^{2,1}$ is independent of $\eps$. 
Then we have the following.

\begin{thm} \label{thm_lower}
Suppose that $N$ is weakly dissipative 
and  that the Fourier transform of $\psi$ does not 
vanish at the point $\xi_0$ coming from {\rm (ii)} in 
Definition~\ref{dfn_wd}.
Then we can choose  $\eps_0>0$ such that 
the global solution $u$ to \eqref{eq}--\eqref{data_eps} 
satisfies
\[
 \liminf_{t\to +\infty}((\log t)^{1/4}\|u(t)\|_{L_x^{2}})>0
\]
for  $\eps \in (0,\eps_0]$.
\end{thm}

Now, let us explain heuristically why $L^2$-decay rate should be 
$O((\log t)^{-1/4})$  if $\hat{\varphi}(\xi_0)\ne 0$. 
For this purpose, let us first remember the fact that the solution $u^0$ 
to the free Schr\"odinger equation (i.e., the case of $N=0$) 
behaves like 
\[
 \pa_x^k u^0(t,x)
 \sim 
 \left(\frac{ix}{t}\right)^k \frac{e^{-i\pi/4}}{\sqrt{t}} 
 \hat{\varphi}\left( \frac{x}{t}\right) e^{i\frac{x^2}{2t}}
 +\cdots
\]
as $t\to+\infty$ for $k=0,1,2,\ldots$. 
Viewing it as a rough approximation of the solution 
$u$ for \eqref{eq}, we may expect that $\pa_x^k u(t,x)$ could be 
better approximated by 
\[
 \left(\frac{ix}{t}\right)^k \frac{1}{\sqrt{t}} 
 A\left(\log t,  \frac{x}{t}\right) e^{i\frac{x^2}{2t}}
\]
with a suitable function $A(\tau,\xi)$, 
where $\tau=\log t$, $\xi=x/t$ and $t\gg 1$. 
Note that 
$A(0,\xi)=e^{-i\pi/4}\, \hat{\varphi}(\xi)$ 
and that the extra variable $\tau=\log t$ is responsible 
for possible long-range nonlinear effect. 
Substituting the above expression into \eqref{eq} 
and keeping only the leading terms, 
we can see (at least formally) that $A(\tau,\xi)$ should satisfy 
the ordinary differential equation 
\begin{align*}
 i\pa_{\tau} A = \nu(\xi)|A|^2A+\cdots 
\end{align*}
under \eqref{weak_gi}. If $N$ is weakly dissipative, we see that 
\[
  \pa_{\tau} |A|^2 = -2c_0(\xi-\xi_0)^2|A|^4+\cdots.
\]
Then it follows that 
\[
 |A(\tau,\xi)|^2
=
\frac{|\hat{\varphi}(\xi)|^2}{1+2c_0(\xi-\xi_0)^2 |\hat{\varphi}(\xi)|^2\tau} 
+\cdots,
\]
whence 
\begin{align*}
\|u(t)\|_{L_x^2}
\sim 
\|A(\log t)\|_{L^2_{\xi}}
\sim 
\left(\int_{\R} 
\frac{|\hat{\varphi}(\xi)|^2}{1+2c_0(\xi-\xi_0)^2|\hat{\varphi}(\xi)|^2
\log t}\, d\xi\right)^{1/2}\quad 
(t\to+\infty).
\end{align*}
By considering the behavior as $t\to +\infty$ of this integral carefully, 
we see that $L^2$-decay rate in the weakly dissipative case should be 
just $O((\log t)^{-1/4})$ if $\hat{\varphi}(\xi_0)\ne 0$.

Our strategy of the proof of Theorems~\ref{thm_upper} and \ref{thm_lower} 
is to justify the above heuristic argument. 
The key is to concentrate on 
$\alpha(t,\xi)=\op{F}[\op{U}(-t)u(t,\cdot)](\xi)$, 
which is expected to play the role of $A(\log t,\xi)$ in the above argument. 
In the previous work \cite{LNSS3}, we have derived 
the pointwise estimate \eqref{est_LNSS3} for $\alpha(t,\xi)$.
Instead of doing so, our proof below will be based on more direct asymptotic 
analysis in the spirit of Hayashi-Naumkin~\cite{HN1}, \cite{HNS}. 

The rest of this paper is organized as follows. 
In Section~\ref{sec_keylemmas}, we establish 
some technical lemmas to make some necessary preparations to obtain our
main theorems. In Section~\ref{sec_proof}, we prove Theorems~\ref{thm_upper} 
and \ref{thm_lower}. In Section~\ref{sec_strict}, we make several remarks on 
the strictly dissipative case. Finally, some backgrounds on $\nu(\xi)$ 
are presented in \ref{sec_app_a}.

\section{Key Lemmas}  \label{sec_keylemmas}
In this section we introduce three lemmas which play key roles in our 
analysis. 
The first one is related to the ordinary differential equation in the form 
\begin{align*}
 i\pa_t \beta(t,\xi) 
 = 
 \frac{\nu(\xi)}{t} |\beta(t,\xi) |^2 \beta(t,\xi)  
 + 
 \cdots,
\end{align*}
which we call {\em the profile equation} associated with \eqref{eq} under 
\eqref{weak_gi}. 
The other two are related to the integral in the form 
\begin{align}
\int_{\R} 
\frac{|\hat{\varphi}(\xi)|^2}{1+2c_0(\xi-\xi_0)^2|\hat{\varphi}(\xi)|^2
\log t}\, d\xi,
\label{integral}
\end{align}
which appears naturally as explained in the introduction.

In what follows we denote several positive constants by the same letter $C$, 
which may vary from one line to another.

\subsection{A lemma related to the profile equation } \label{subsec_key1}

We start with the following lemma which characterizes the large-time 
asymptotics of solutions to the profile equation.
\begin{lem} \label{lem_asymp_prof}
Let $\theta_0(\xi)$, $\mu(\xi)$ be $\C$-valued continuous functions on $\R$ 
satisfying 
\begin{align}
|\theta_0(\xi)|\le C\eps\jb{\xi}^{-2}, 
\quad 
|\mu(\xi)|\le C\jb{\xi}^3, 
\quad 
\imagpart \mu(\xi) \le 0,
\label{assump_2_1_1}
\end{align}
where $\eps>0$ is a small parameter. 
Let $\rho(t,\xi)$ be a $\C$-valued continuous function on 
$[1,\infty)\times \R$ satisfying 
\begin{align}
 |\rho(t, \xi)| 
\le  \frac{C\eps^3}{\jb{\xi}^2t^{1+\kappa}}
\label{assump_2_1_2}
\end{align}
with some $\kappa>0$. 
If the function $\beta:[1,\infty)\times\R\to \C$ solves 
\begin{align}
 i\pa_t \beta(t,\xi) 
 = 
 \frac{\mu(\xi)}{t} |\beta(t,\xi) |^2 \beta(t,\xi)  
 + 
 \rho(t,\xi), 
 \quad
 \beta(1,\xi) = \theta_0(\xi)
 \label{profile_eq}
\end{align}
and $\eps>0$ is suitably small,
then we have 
\begin{align*}
 |\beta(t,\xi) - A(\log t,\xi)|
\le \frac{C\eps^3}{\jb{\xi}^2t^{\kappa-\delta}}
\end{align*}
for $(t,\xi)\in [1,\infty)\times \R$, 
where $\delta$ is an arbitrarily small positive real number, 
and the function $A:[0,\infty)\times\R\to \C$ solves
\[
i\pa_\tau A(\tau,\xi) 
 = 
 \mu(\xi) |A(\tau,\xi) |^2 A(\tau,\xi), 
 \quad
 A(0,\xi) = \theta_{0}(\xi)+\theta_1(\xi)
\]
with some $\theta_1:\R\to \C$ satisfying 
$|\theta_1(\xi)| \leq C \eps^3\jb{\xi}^{-2}$ .
\end{lem}
\begin{proof} 
Many parts of the argument below are similar to those given in \cite{Su} and 
\cite{HNS}, but we must be more careful in several parts. 
Our point of departure is the fact that $\beta(t,\xi)$ admits the 
decomposition 
\[
 \beta(t, \xi)= \frac{ P(t,\xi) }{ \sqrt{Q(t,\xi)} },
\]
where $P:[1,\infty)\times \mathbb{R} \rightarrow \mathbb{C}$ and 
$Q: [1,\infty)\times \mathbb{R} \rightarrow \mathbb{R}^{+}$ satisfy
\begin{align}
 \left\{
 \begin{array}{l}
  \dis{\pa_t P(t,\xi) = -i \frac{\realpart \mu(\xi) }{t} 
       \frac{ \bigl| P(t,\xi) \bigr|^2 }{ Q(t,\xi) } P(t,\xi) 
       -i\sqrt{Q(t, \xi) } \rho(t, \xi) ,
      }\\[4mm]
 \dis{\pa_t Q(t,\xi) =-2 \frac{ \imagpart \mu(\xi) }{t} 
       \bigl| P(t,\xi)  \bigr|^2,
       }\\[4mm]
  P(1, \xi) = \theta_0(\xi), 
  \quad 
  Q(1,\xi) = 1.
 \end{array}
 \right.
\label{eq_pq}
\end{align}
This is a result of uniqueness for the solution to \eqref{profile_eq}. 
We are going to specify 
the asymptotics of $P(t,\xi) $ and $Q(t,\xi) $ as $t \to +\infty$. 
We first show that there exists an $\eps_1>0$ 
such that 
\begin{align}
  \sup_{(t,\xi) \in [1,\infty) \times \R}\jb{\xi}^2| P(t,\xi)|  
  <  2C_0\eps
  \label{est_p}
\end{align}
if $\eps \in (0,\eps_1]$, where $C_0>0$ is a constant which satisfies
\[
\sup_{\xi \in \R}\jb{\xi}^2|\theta_0(\xi)| \le C_0 \eps.
\]
We shall argue by contradiction: 
If this is not true, there exists $T_{\eps} \in (1,\infty)$ 
such that 
\[
 \sup_{(t,\xi) \in [1, T_{\eps})\times \R} \jb{\xi}^2|P(t,\xi)|  
 \le  
 2C_0\eps
 \quad \mbox{and} \quad 
 \sup_{\xi \in \R}\jb{\xi}^2|P(T_{\eps},\xi)|  =  2C_0\eps.
\]
Then, by integrating $\pa_t Q$ in \eqref{eq_pq} 
with respect to the time variable from 1 to $t$, we have 
\begin{equation}
   1 \le Q(t,\xi) \le 
   1 - 2 \frac{\imagpart \mu(\xi)}{\jb{\xi}^4} (2C_0)^2\eps^2 \log t
   \label{2.6}
\end{equation}
for $t \in [1,T_{\eps}]$ and $\xi\in \R$. 
On the other hand, it follows from the first equation of \eqref{eq_pq} that 
\begin{align}
   \pa_t \Bigl(|P(t,\xi)|^2 \Bigr) \notag
 &= 
   2 \realpart 
     \Bigl( \cc{P(t,\xi)} \pa_t  P(t,\xi) \Bigr) \notag\\
 &= 
   2 \realpart
     \Bigl(-i\cc{P(t,\xi)} \sqrt{Q(t,\xi)} \rho(t,\xi) \Bigr)\notag\\
 &\le 
   2 \bigl| P(t,\xi) \bigr| \, \bigl| \sqrt{Q(t,\xi)} \rho(t,\xi) \bigr|. 
\label{2.7}
\end{align}
By \eqref{assump_2_1_1}, \eqref{assump_2_1_2}, \eqref{2.6} and 
\eqref{2.7}, we have 
\begin{align*}
  \jb{\xi}^2|P(T_{\eps},\xi)|
 &\leq 
  \jb{\xi}^2|\theta_0(\xi)| 
  + 
  \jb{\xi}^2\int_{1}^{T_{\eps}} 
\bigl| \sqrt{Q(\tau,\xi)} \rho(\tau,\xi) \bigr|\, d\tau
\\
 &\le 
  C_0\eps + 
  \int_{1}^{\infty} 
  (1+ C\varepsilon^2 \log \tau)^{1/2} 
    \frac{C\eps^3}{\tau^{1+\kappa}} d\tau\\
 &\le 
  C_0\eps + C_1\eps^3
\end{align*}
with some $C_1>0$ independnt of $\eps \in (0,1]$. 
If we choose $\eps_1=\min \{1, \sqrt{C_0/(2C_1)}\}$, we have 
\[
 \sup_{\xi \in \R} \jb{\xi}^2|P(T_{\eps},\xi)| 
\le C_0\eps + \frac{C_0}{2}\eps  
<  2C_0\eps 
\] 
for $\eps \in (0,\eps_1]$, which is the desired contradiction. 
Hence \eqref{est_p} must hold if $\eps$ is suitably small. 
Also we see from \eqref{2.6} and \eqref{assump_2_1_1} that 
\[
  1 \le  Q(t,\xi)  \le 1 + C  \frac{\eps^2}{\jb{\xi}}\log t
\]
for  $t \ge 1$, $\xi \in \R$. 
Next we define 
\[
 \Psi(t,\xi)
 = 
 \int_{1}^{t}
 \frac{ \realpart \mu(\xi) |P(\tau,\xi)|^2 }{Q(\tau,\xi)} 
 \frac{ d \tau }{\tau}
\]
so that 
$\pa_t \left(P(t,\xi) e^{i\Psi(t,\xi)} \right)
 = -i\sqrt{Q(t,\xi)} \rho(t,\xi) e^{i\Psi(t,\xi)}$. 
Note that 
\begin{align*}
  \left| \sqrt{Q(t,\xi)} \rho(t,\xi) \right|
 \leq
  (1+ C\varepsilon^2 \log t)^{1/2} \frac{C\eps^3}{\jb{\xi}^2t^{1+\kappa}} 
 \le
  \frac{C \eps^3}{\jb{\xi}^2t^{1+\kappa-\delta}},
\end{align*}
where $\delta>0$ can be arbitrarily small. So we obtain
\begin{align}
|P(t,\xi)- P_{\infty}(\xi)e^{-i \Psi(t,\xi)}| 
&=
\left|
i e^{-i\Psi(t,\xi)}\int_t^{\infty}
\sqrt{Q(\tau,\xi)}\rho(\tau,\xi)e^{i\Psi(\tau,\xi)}d\tau
\right|
 \notag\\
&\le 
\frac{C \eps^3}{\jb{\xi}^2t^{\kappa -\delta}}, 
\label{2.9}
\end{align}
where 
\[
 P_{\infty}(\xi) 
 = 
 \theta_0(\xi) 
 - 
 i\int_{1}^{\infty} 
 \sqrt{Q(\tau,\xi)} \rho(\tau,\xi) e^{i\Psi(\tau,\xi)}\, d\tau, 
\quad 
\xi \in \mathbb{R}.
\]
Note that
\begin{align}
|P_{\infty}{(\xi)}|
\leq 
|\theta_0({\xi})|+\int_1^{\infty}|\sqrt{Q(t,\xi)}\rho(t,\xi)|d\tau 
\leq 
C \eps {\jb{\xi}^{-2}}.
\label{2.10}
\end{align}
We also set 
$Q_{\infty}(t,\xi)= 1 - 2 {\imagpart \mu(\xi)} 
 \bigl(|P_{\infty}(\xi)|^2 \log t + \Lambda (\xi) \bigr)$ 
with
\[
  \Lambda(\xi) 
 =
  \int_{1}^{\infty} 
  \bigl( |P(\tau,\xi)|^2- |P_{\infty}(\xi)|^2 \bigr)
  \frac{d\tau}{\tau}.
\]
From \eqref{est_p}, \eqref{2.9} and \eqref{2.10} it follows that
\begin{align}
  \bigl| |P(t,\xi)|^2- |P_{\infty}(\xi)|^2 \bigr| 
 &\le 
  \bigl| P(t,\xi)- P_{\infty}(\xi) e^{-i \Psi(t,\xi)} \bigr| 
  \bigl( |P(t,\xi)|+|P_{\infty}(\xi)| \bigr)\notag\\
 &\le 
 \frac{C \eps^4}{ \jb{\xi}^4t^{\kappa -\delta}},\label{2.12}
\end{align}
which implies
\begin{align}
 \left| Q(t,\xi) - Q_{\infty}(t,\xi) \right|
&=
2\left|\imagpart\mu(\xi)
\int_t^{\infty}(|P(\tau,\xi)|^{2}-|P_{\infty}(\xi)|^{2})\frac{d\tau}{\tau}
\right|\notag\\
&\le C\jb{\xi}^3
   \int_{t}^{\infty} 
   {\bigl| |P(\tau,\xi)|^2 - |P_{\infty}(\xi)|^2 \bigr|}\frac{d\tau}{\tau}\notag\\
& \le  \frac{C \eps^4}{ \jb{\xi}t^{\kappa -\delta}}.
\label{2.130}
\end{align}
We remark that  \eqref{2.12} also gives us
\begin{align}
Q_{\infty}(t,\xi)
&\ge
1-2|\mu(\xi)| |\Lambda(\xi)|
 \notag\\
&\ge
1-C\jb{\xi}^3 \int_{1}^{\infty}\frac{C\eps^4}{\jb{\xi}^4\tau^{\kappa-\delta}}  
\frac{d\tau}{\tau}
\notag\\
&\ge
1-\frac{C\eps^4}{\jb{\xi}}
\notag\\
&\geq
\frac{1}{2}
\label{2.11a}
\end{align}
for $t\ge 1$ and $\xi \in \R$, provided that $\eps>0$ is suitably small. 
Let us also introduce 
\[
   \Phi(t,\xi)
  =
   \frac{\realpart \mu(\xi)}{t} 
   \left(\frac{|P(t,\xi)|^2}{Q(t,\xi)}
          - 
         \frac{|P_{\infty}(\xi)|^2}{Q_{\infty}(t,\xi)}
   \right).
\]
Then we have 
\begin{align*}
  \Psi(t,\xi)
  &= 
 |P_{\infty}(\xi)|^2 \realpart \mu(\xi) 
 \int_{1}^{t} \frac{1}{Q_{\infty}(\tau,\xi)}\frac{d\tau}{\tau}
 +
 \int_{1}^{t}\Phi(\tau,\xi)d \tau \\
  &=
   |P_{\infty}(\xi)|^2 \realpart \mu(\xi) 
     \int_{0}^{\log t} \frac{d \sigma}{Q_{\infty}(e^{\sigma},\xi)}
   +
    \int_{1}^{\infty} \Phi(\tau,\xi)\, d \tau
   -
    \int_{t}^{\infty} \Phi(\tau,\xi)\, d \tau
\end{align*}
and  
\begin{align*}
  \int_{t}^{\infty} \left| \Phi(\tau,\xi) \right|\, d\tau
   &=
   \int_{t}^{\infty}\left|\realpart \mu(\xi)
  \biggl(
         \frac{ |P(\tau,\xi)|^2-|P_{\infty}(\xi)|^2 }{Q(\tau,\xi)}
         -
         \frac
         { |P_{\infty}(\xi)|^2 \bigl(Q(\tau,\xi)-Q_{\infty}(\tau,\xi) \bigr) }
         { Q(\tau,\xi)Q_{\infty}(\tau,\xi) }
  \biggr)\right|
  \frac{1}{\tau}d\tau\\
 &\le
  C \jb{\xi}^3\int_{t}^{\infty}
  \biggl(
         \frac{ \bigl||P(\tau,\xi)|^2-|P_{\infty}(\xi)|^2 \bigr| }{Q(\tau,\xi)}
         +
         \frac
         { |P_{\infty}(\xi)|^2 \bigl|Q(\tau,\xi)-Q_{\infty}(\tau,\xi) \bigr| }
         { Q(\tau,\xi)Q_{\infty}(\tau,\xi) }
  \biggr)
  \frac{d\tau}{\tau}\\
 &\le 
  \frac{C \eps^4}{\jb{\xi}} 
    \int_{t}^{\infty} \frac{d\tau}{\tau^{1+\kappa -\delta}}\\
 &\le
  \frac{C \eps^4}{\jb{\xi} t^{\kappa -\delta}},
\end{align*}
whence 
\begin{align}
 &\left|  \exp(-i\Psi(t,\xi)) \exp
 \Bigl(i \int_{1}^{\infty} \Phi(\tau,\xi)\, d \tau\Bigr)
 - \exp
  \biggl(
    -i |P_{\infty}(\xi)|^2 {\realpart \mu(\xi)} 
    \int_{0}^{\log t} \frac{d \sigma}{Q_{\infty}(e^{\sigma},\xi)}  
  \biggr)
\right| \notag\\
&= 
\left|\exp\left( -i |P_{\infty}(\xi)|^2 {\realpart \mu(\xi)} \int_{0}^{\log t} \frac{d \sigma}{Q_{\infty}(e^{\sigma},\xi)} \right)\right|
\cdot
\left|\exp{\left(i\int_t^{\infty} \Phi(\tau,\xi)d\tau\right)}-1\right|
 \notag  \\
&\le
1\cdot \int_{t}^{\infty} \left| \Phi(\tau,\xi) \right|\, d\tau
 \notag  \\
 &\le
  \frac{C\eps^4}{\jb{\xi} t^{\kappa-\delta}}.\label{2.13}
\end{align}
Therefore, putting  
$\theta_{\infty}(\xi) =  P_{\infty}(\xi) \exp
 \left(-i \int_{1}^{\infty} \Phi(\tau,\xi)\, d \tau\right)$, 
we deduce from \eqref{2.9}, \eqref{2.10}, \eqref{2.13} that
\begin{align*}
  P(t,\xi)
 =& 
  P_{\infty}(\xi) e^{-i\Psi(t,\xi)} 
 + 
  O(\eps^3 t^{-\kappa+\delta}\jb{\xi}^{-2})\\
 =&
 \theta_{\infty}(\xi)
  \exp\biggl(
        -i |\theta_{\infty}(\xi)|^2 {\realpart \mu(\xi)} 
        \int_{0}^{\log t} \frac{d \sigma}{Q_{\infty}(e^{\sigma},\xi)}  
        \biggr) 
  +\theta_{\infty}(\xi) \cdot O(\eps^4t^{-\kappa+\delta}\jb{\xi}^{-1})\\
  & +O(\eps^3 t^{-\kappa+\delta}\jb{\xi}^{-2}) \\
 =&
  \theta_{\infty}(\xi) \exp
  \biggl(
         -i |\theta_{\infty}(\xi)|^2 {\realpart \mu(\xi)} 
         \int_{0}^{\log t} \frac{d \sigma}{Q_{\infty}(e^{\sigma},\xi)}  
         \biggr) 
   + O(\eps^3 t^{-\kappa+\delta}\jb{\xi}^{-2}) 
\end{align*}
and from \eqref{2.130}, \eqref{2.11a} that
\begin{align*} 
  \frac{1}{\sqrt{Q(t,\xi)}}
 &=
  \frac{1}{\sqrt{Q_{\infty}(t,\xi)}}
  +
  \frac
  {Q_{\infty}(t,\xi)-Q(t,\xi)}
  { 
    \sqrt{Q(t,\xi)Q_{\infty}(t,\xi)}
    ( \sqrt{Q(t,\xi)}+ \sqrt{Q_{\infty}(t,\xi)} )
  }\\
 &=
  \frac{1}{\sqrt{1 - 2\imagpart \mu(\xi)
            \bigl(|\theta_{\infty}(\xi)|^2 \log t + \Lambda(\xi)\bigr)}}
 + O(\eps^4\jb{\xi}^{-1}t^{-\kappa +\delta}).
\end{align*}
Piecing them together, we have 
\begin{align} 
  \beta(t,\xi)
 &=
  \frac{P(t,\xi)}{\sqrt{Q(t,\xi)}} \notag\\
 &=
  \frac{\theta_{\infty}(\xi) \exp
     \biggl(
      -i |\theta_{\infty}(\xi)|^2 \realpart \mu(\xi) 
        \int_{0}^{\log t} \frac {d \sigma}
                   {1 - 2 \imagpart \mu(\xi) 
                  \bigl(|\theta_{\infty}(\xi)|^2 \sigma + \Lambda(\xi) \bigr)
                   }  
     \biggr)}
     {\sqrt{1 - 2\imagpart \mu(\xi)
            \bigl(|\theta_{\infty}(\xi)|^2 \log t + \Lambda(\xi)\bigr)
           }}
 + O(\eps^3 t^{-\kappa+\delta}\jb{\xi}^{-2})
\label{goal}
\end{align}
Finally we set
\[
 A(\tau,\xi)=
\frac{\theta_{\infty}(\xi) \exp
     \biggl(
      -i |\theta_{\infty}(\xi)|^2 \realpart \mu(\xi) 
       \int_{0}^{\tau} \frac {d \sigma}
        {1 - 2 \imagpart \mu(\xi) 
         \bigl(|\theta_{\infty}(\xi)|^2 \sigma + \Lambda(\xi) \bigr)}  
     \biggr)}
     {\sqrt{1 - 2\imagpart \mu(\xi)
            \bigl(|\theta_{\infty}(\xi)|^2 \tau + \Lambda(\xi)\bigr)
           }}.
\]
Then \eqref{goal} can be rewritten as 
\begin{align*}
 |\beta(t,\xi) - A(\log t,\xi)|
\le \frac{C\eps^3}{\jb{\xi}^2t^{\kappa-\delta}}.
\end{align*}
Also straightforward calculations give us
\begin{align*}
 i\pa_\tau A(\tau,\xi)
 =&
 \mu(\xi)|A(\tau,\xi)|^2A(\tau,\xi) 
\end{align*}
and
\begin{align*}
 A(0,\xi)
 =&
 \frac{\theta_{\infty}(\xi)}
      {\sqrt{1-2\imagpart \mu(\xi)\Lambda(\xi)}}\\
 =&
 \left(
 \theta_0(\xi) -i\int_1^{\infty} 
  \sqrt{Q(\tau,\xi)}\rho(\tau,\xi) e^{i\Psi(\tau,\xi)\,}d\tau
 \right)
 \frac{1+(e^{-i\int_1^{\infty} \Phi(\tau,\xi)\, d\tau}-1)}
      {\sqrt{1-2\imagpart \mu(\xi)\Lambda(\xi)}}\\
 =&
 \bigl(\theta_0(\xi)+O(\eps^3\jb{\xi}^{-2})\bigr)
 \frac{1+O(\eps^4\jb{\xi}^{-1})}{\sqrt{1+O(\eps^4\jb{\xi}^{-1})}}\\
 =&
 \theta_0(\xi)+O(\eps^3\jb{\xi}^{-2}).
\end{align*}
\end{proof}

\subsection{Lemmas related to the integral \eqref{integral}}  
\label{subsec_key2}
In this subsection, we shall derive upper and lower bounds for 
the integral 
\begin{align}
S(\tau)
= \int_{\R} \frac{|\theta(\xi)|^2}{1+(\xi-\xi_0)^2|\theta(\xi)|^2\tau}\, d\xi,
\quad \tau\ge 1.
\label{dfn_S}
\end{align}
These results can be used to investigate estimates on 
$\|A(\log t)\|_{L^2_{\xi}}$ later.
\begin{lem}\label{lem_int1}
Let $\theta \in L^{\infty}(\R)$ and $\xi_0\in \R$. 
For the integral $S(\tau)$ given by \eqref{dfn_S}, 
we have 
\[
 S(\tau) \le 4\|\theta\|_{L^{\infty}}\tau^{-1/2},
\quad \tau \ge 1.
\]
\end{lem}
\begin{proof} 
It suffices to consider the case of $\|\theta\|_{L^{\infty}}>0$.
We put $a=\|\theta\|_{L^{\infty}}^2$ and 
\[
 S_1(\tau)= 
\int_{|\xi-\xi_0|\le m \tau^{-1/2}} 
\frac{|\theta(\xi)|^2}{1+(\xi-\xi_0)^2|\theta(\xi)|^2\tau}\, d\xi,
\quad \tau \ge 1,
\]
with a constant $m>0$ which is to be fixed. It follows that 
\[
 S_1(\tau) 
 \le 
 \int_{\xi_0-m\tau^{-1/2}}^{\xi_0+m\tau^{-1/2}} 
\frac{a}{1+0}\, d\xi
 = 
 2m a\tau^{-1/2}.
\]
We also set $S_2(\tau)=S(\tau)-S_1(\tau)$. Since 
\[
\frac{|\theta(\xi)|^2}{1+(\xi-\xi_0)^2|\theta(\xi)|^2\tau}
 \le \frac{1}{(\xi-\xi_0)^2\tau},
\]
we have 
\[
 S_2(\tau) \le
 \frac{2}{\tau} \int_{\xi_0+m\tau^{-1/2}}^{\infty} 
 \frac{d\xi}{(\xi-\xi_0)^2}
 =
 \frac{2}{\tau} \cdot \frac{1}{m \tau^{-1/2}}
 =\frac{2}{m} \tau^{-1/2}.
\]
By choosing $m=a^{-1/2}$ (that is, $m a =1/m$), we arrive at 
\[
 S(\tau)
 \le 
 2\left(m a+ \frac{1}{m}\right) \tau^{-1/2}
 =
 4\|\theta\|_{L^{\infty}}\tau^{-1/2}.
\]
\end{proof}
\begin{lem}\label{lem_Sagawa} 
Let $\theta$, $\xi_0$ and $S(\tau)$ be as in Lemma~\ref{lem_int1}. 
Assume that there exists an open interval $I$ with $I\ni \xi_0$ such that 
$\inf_{\xi \in I}|\theta(\xi)|>0$.
Then we can choose a positive constant $C_2$, which is independent of 
$\tau \ge 1$ but may depend on $\theta$ and $\xi_0$, such that 
\[
 S(\tau) \ge C_2\tau^{-1/2}, \quad \tau \ge 1.
\]
\end{lem}
\begin{proof} 
As before we put $a=\|\theta\|_{L^{\infty}}^2$. We also set 
$b=\inf_{\xi \in I}|\theta(\xi)|^2$ so that 
\[
 \frac{|\theta(\xi)|^2}{1+(\xi-\xi_0)^2|\theta(\xi)|^2\tau}
 \ge 
 \frac{b}{1+(\xi-\xi_0)^2 a \tau},
 \quad \xi \in I.
\]
We take $m>0$ so small that 
$[\xi_0-m, \xi_0+m]\subset I$. 
Then, by the change of variable $\eta=(\xi-\xi_0)\tau^{1/2}$, 
we see that 
\begin{align*}
S(\tau) 
\ge 
\int_{\xi_0-m \tau^{-1/2}}^{\xi_0+m \tau^{-1/2}}
\frac{b}{1+(\xi-\xi_0)^2a\tau}\, d\xi
=
\int_{-m}^{m} \frac{b}{1+a\eta^2}\,\tau^{-1/2}d\eta,
\end{align*}
which yields the desired lower estimate with 
\[
C_2= \int_{-m}^{m} \frac{b}{1+a\eta^2}\,d\eta.
\]
\end{proof}

\section{Proof of the main results}  \label{sec_proof}
In this section we are going to prove Theorems \ref{thm_upper} and 
\ref{thm_lower}. First we make a reduction of the original equation \eqref{eq} 
to the profile equation. Then we will apply the lemmas prepared in 
the previous section to reach the main results. 

\subsection{Reduction to the profile equation}  
\label{subsec_reduction}
The argument in this subsection is exactly the same as that given in 
\cite{LNSS3}. 
We write $\op{L}=i\pa_t +\frac{1}{2}\pa_x^2$ and $\op{J}=x+it\pa_x$.  
Important relations are $[\pa_x, \op{J}]=1$, $[\op{L}, \op{J}]=0$, 
where $[\cdot, \cdot]$ denotes the commutator. 
Next we set $\alpha(t,\xi)=\op{F}[\op{U}(-t)u(t,\cdot)](\xi)$ for 
the solution $u(t,x)$ to \eqref{eq}, where 
$\op{U}(t)=\exp(i\frac{t}{2}\pa_x^2)$ and 
\[
 \op{F}\phi(\xi)=\hat{\phi}(\xi)
=\frac{1}{\sqrt{2\pi}} \int_{\R} \phi(y)e^{-iy\xi}\, dy.
\]
By virtue of the previous works \cite{SagSu} and \cite{LNSS3}, 
We have  already known the following.
\begin{lem}
\label{lem3.1}
Let $\eps=\|\varphi\|_{H^{3}\cap H^{2,1}}$ be suitably small. 
Assume that \eqref{weak_gi} and $({\bf A})$ are fulfilled. Then 
the solution $u$ to \eqref{eq}--\eqref{data} satisfies 
\begin{align}
 |\alpha(t,\xi)| \le \frac{C\eps}{\jb{\xi}^2}
\label{est_alpha}
\end{align}
for $t\ge 0$, $\xi \in \R$, and 
\begin{align}
\|u(t)\|_{H^3}+\|\op{J}u(t)\|_{H^2}\le C\eps (1+t)^{\gamma}
\label{est_sobolev}
\end{align}
for $t\ge 0$, where $0<\gamma<1/12$. 
Moreover, we have 
$\alpha(t,\xi)=\beta(t,\xi)+\sigma(t,\xi)$
with 
\[
  |\sigma(t,\xi)| \le \frac{C\eps^3}{t^{1/2}\jb{\xi}^4}
\]
and
\[
 i\pa_t \beta(t,\xi)  
 =
 \frac{\nu(\xi)}{t} |\beta(t,\xi)|^2 \beta(t,\xi)+R(t,\xi),
 \quad 
 |R(t,\xi)|
\le
  \frac{C \eps^3}{t^{1+\kappa}\jb{\xi}^2}
\]
for $t\ge 1$, where $0<\kappa <1/4$. 
\end{lem}
The proof can be found in Section~2 of \cite{LNSS3} or Section~4 of 
\cite{SagSu}, so we skip it here.

\subsection{Proof of Theorem \ref{thm_upper}}  \label{subsec_proof_of_thm_1}
We are ready to prove Theorem \ref{thm_upper}. 
First we consider the easier case $0\le t \le e$. 
It follows from \eqref{est_sobolev} that 
\[
\|u(t)\|_{L_x^2} 
\le 
C\eps(1+t)^{\gamma} 
\left(\frac{1+\eps^2\log(t+1)}{1+\eps^2\log(t+1)}\right)^{1/4}
\le 
\frac{C\eps}{(1+\eps^2\log(t+1))^{1/4}}.
\]
Let us turn to the main case $t\ge e$ (i.e., $\log t\ge 1$). 
By \eqref{est_alpha} and the $L^2$-unitarity of $\op{F}$, $\op{U}(t)$, 
we have
\[
 \|u(t)\|_{L^2_x}=\|\alpha(t)\|_{L^2_{\xi}} \le C\eps.
\]
Also we see from Lemma~\ref{lem3.1} that there exists 
a function $\beta(t,\xi)$ satisfying
\begin{align}
 \|\alpha(t)-\beta(t)\|_{L^2_{\xi}} 
\leq
\left(\int_{\R} \frac{C\eps^6}{t\jb{\xi}^8}\, d\xi\right)^{1/2}
= \frac{C\eps^3}{t^{1/2}}
\label{alpha-beta}
\end{align}
and 
\[
 i\pa_t \beta(t,\xi)  
 =
 \frac{\nu(\xi)}{t} |\beta(t,\xi)|^2 \beta(t,\xi)+R(t,\xi)
\]
with 
\[
 |R(t,\xi)|
\le
  \frac{C \eps^3}{t^{1+\kappa}\jb{\xi}^2}, 
\quad 
|\beta(1, \xi)|\leq C\eps \jb{\xi}^{-2}.
\]

Now we apply Lemma~\ref{lem_asymp_prof}. Then we find 
$A:[0,\infty)\times\R\to \C$ satisfying
\begin{align}
 \|\beta(t)-A(\log t)\|_{L_{\xi}^2} \le \frac{C\eps^3}{t^{\kappa-\delta}}
 \label{beta-A}
\end{align}
and
\begin{align*}
i\pa_\tau A(\tau,\xi) 
 = 
 \nu(\xi) |A(\tau,\xi) |^2 A(\tau,\xi), 
 \quad
 A(0,\xi) = \beta(1,\xi)+\theta_1(\xi)
\end{align*}
with some $\theta_1:\R\to \C$ such that 
$|\theta_1(\xi)| \leq C \eps^3\jb{\xi}^{-2}$. 
Since  $N$ is weakly dissipative, we have 
\begin{align*}
 \pa_\tau |A(\tau,\xi)|^2 
 &=
 2 \imagpart\!\left(\cc{A(\tau,\xi)}i\pa_{\tau}A(\tau,\xi)\right)\\
 &=
 2 \imagpart \nu(\xi)\,  |A(\tau,\xi)|^4\\
 &= 
 -2c_0(\xi-\xi_0)^2 (|A(\tau,\xi)|^2)^2, 
\end{align*}
which leads to
\[
 |A(\tau,\xi)|^2
=
\frac{|A(0,\xi)|^2}{ 1+2c_0(\xi-\xi_0)^2 |A(0,\xi)|^2 \tau}.
\]
By Lemma~\ref{lem_int1} with $\tau=\log t$ and 
$\theta(\xi)=\sqrt{2c_0}A(0,\xi)$, 
we obtain 
\begin{align}
 \|A(\log t)\|_{L_{\xi}^2}^2 
&=
\frac{1}{2c_0}
\int_{\R} \frac{2c_0|A(0,\xi)|^2}{ 1+2c_0(\xi-\xi_0)^2 |A(0,\xi)|^2 \log t}
 \, d\xi \notag
\\
&\le 
\frac{4(\|\beta(1)\|_{L_{\xi}^{\infty}}+\|\theta_1\|_{L_{\xi}^{\infty}})}
{\sqrt{2c_0} (\log t)^{1/2}} \notag
\\
&\le 
\frac{C\eps}{(\log t)^{1/2}}
\label{Estimate A}
\end{align}
for $t\ge e$. By \eqref{Estimate A}, \eqref{beta-A} and \eqref{alpha-beta}, 
we arrive at 
\begin{align*}
\|u(t)\|_{L^2_x}
&=
\|\alpha(t)\|_{L^2_{\xi}}\\
&\le
\|A(\log t)\|_{L^2_{\xi}}
+\|A(\log t)-\beta(t)\|_{L^2_{\xi}}
+\|\beta(t)-\alpha(t)\|_{L^2_{\xi}}\\
&\le
\frac{C\eps^{1/2}}{(\log t)^{1/4}}
+\frac{C\eps^3}{t^{\kappa-\delta}} +\frac{C\eps^3}{t^{1/2}}\\
&\le
\frac{C\eps}{(\eps^2 \log t)^{1/4}},
\end{align*}
whence
\[
(1+\eps^2\log(t+1))^{1/4}\|u(t)\|_{L^2_x}
\le
\|u(t)\|_{L^2_x}+ C(\eps^{2}\log t)^{1/4}\|u(t)\|_{L^2_x}
\le
C\eps
\]
for $t\ge e$. This completes the proof of Theorem~\ref{thm_upper}.
\qed
\subsection{Proof of Theorem \ref{thm_lower}}  \label{subsec_proof_of_thm_2}
In order to prove Theorem \ref{thm_lower}, we need one more lemma. 
\begin{lem} \label{asymp_alpha_1}
We put $\alpha(t,\xi)=\op{F}\bigl[\op{U}(-t) u (t,\cdot)\bigr](\xi)$ 
for solution $u$ to \eqref{eq}--\eqref{data_eps} with small $\eps$. 
Then we have 
\[
 |\alpha(1,\xi)-\eps \hat{\psi}(\xi)| \le \frac{C\eps^{2}}{\jb{\xi}^2}, 
\quad \xi \in \R.
\]
\end{lem}
\begin{proof} We follow the method used in \cite{SagSuC}. 
We set 
\[
 Y=\sup_{\xi \in \R} \jb{\xi}^2|\alpha(1,\xi)-\eps \hat{\psi}(\xi)|.
\]
By the inequality 
$\|\hat{\phi}\|_{L^{\infty}}^2
\le 2\|\hat{\phi}\|_{L^2}\|\pa_{\xi}\hat{\phi}\|_{L^2}
\le 2\|\phi\|_{L^2}\|\phi\|_{H^{0,1}}$ 
and the relation $\op{J}=\op{U}(t)x\op{U}(t)^{-1}$, 
we have
\begin{align*}
Y^2
&\le 
C \|\op{U}(1)^{-1}u(1,\cdot)-u(0,\cdot)\|_{H^2_x}\cdot
\|\op{U}(1)^{-1}u(1,\cdot) - u(0,\cdot) \|_{H^{2,1}_x}\\
&\le 
C \int_0^{1} 
 \bigl\|\op{U}(t)^{-1}N(u(t),u_x(t)) \bigr\|_{H^2_x} dt
 \cdot 
\bigl(\|\op{J}u(1)\|_{H^2_x}+\|u(0)\|_{H^{2,1}_x}\bigr)\\
&\le
 C\bigl(\sup_{t\in [0,1]}\|u(t)\|_{H^{3}_x}\bigr)^3
 \cdot C\eps\\
&\le
 C\eps^4,
\end{align*}
whence 
\[
 |\alpha(1,\xi)-\eps \hat{\psi}(\xi)| 
\le 
\frac{Y}{\jb{\xi}^2}
\le
\frac{C\eps^{2}}{\jb{\xi}^2}.
\]
\end{proof}

Now we are in a position to finish the proof of Theorem~\ref{thm_lower}. 
It follows from  Lemma~\ref{asymp_alpha_1} that 
\begin{align*}
 |\alpha(1,\xi_0)|
 \ge 
 \eps|\hat{\psi}(\xi_0)| -C\eps^2
\ge 
 \frac{\eps|\hat{\psi}(\xi_0)|}{2}
\end{align*}
if $\eps>0$ is suitably small. 
By the continuity of $\xi\mapsto \alpha(1,\xi)$, we can choose an 
open interval $I$ with $I\ni \xi_0$ such that 
\[
 \inf_{\xi \in I}|\alpha(1,\xi)| \ge \frac{\eps|\hat{\psi}(\xi_0)|}{3}>0.
\]
Now let $A(\tau,\xi)$ be as in the previous subsection. Then
by Lemmas \ref{lem_asymp_prof} and \ref{lem3.1} we have 
\begin{align*}
 \inf_{\xi \in I}|A(0,\xi)|
 &\ge 
 \inf_{\xi \in I}|\beta(1,\xi)| -C\eps^3\\
 &\ge 
 \inf_{\xi \in I}|\alpha(1,\xi)| -\sup_{\xi \in \R}|\sigma(1,\xi)|-C\eps^3\\
 &\ge
 \frac{\eps|\hat{\psi}(\xi_0)|}{3} -C\eps^3\\
 &\ge 
 \frac{\eps|\hat{\psi}(\xi_0)|}{4}>0
\end{align*}
and 
\[
|A(0,\xi)|\leq|\beta(1,\xi)|+|\theta_1(\xi)|\leq C\eps,
\]
if $\eps>0$ is suitably small. 
So Lemma~\ref{lem_Sagawa} gives us 
\begin{align*}
 \|A(\log t)\|_{L^2_{\xi}}^2 
=
\frac{1}{2c_0}
\int_{\R} \frac{2c_0|A(0,\xi)|^2}{ 1+2c_0(\xi-\xi_0)^2 |A(0,\xi)|^2 \log t}
 \, d\xi
\ge 
\frac{C_2}{(\log t)^{1/2}} 
\end{align*}
with some $C_2>0$. Therefore,
by Lemmas \ref{lem_asymp_prof} and \ref{lem3.1} we obtain 
\begin{align*}
\|u(t)\|_{L^2_x}
&=
\|\alpha(t)\|_{L^2_{\xi}}\\
&\ge
\|A(\log t)\|_{L^2_{\xi}}
-\|A(\log t)-\beta(t)\|_{L^2_{\xi}}
-\|\sigma(t)\|_{L^2_{\xi}}\\
&\ge
\frac{\sqrt{C_2}}{(\log t)^{1/4}}-\frac{C\eps^2}{t^{\kappa-\delta}}\\
&\ge
\frac{\sqrt{C_2}}{2(\log t)^{1/4}}
\end{align*}
for sufficiently large $t$, whence
\begin{align*}
\liminf_{t\to \infty}(\log t)^{1/4}\|u(t)\|_{L^2_x}
\ge \frac{\sqrt{C_2}}{2}>0,
\end{align*}
as desired. 
\qed

\section{Remarks on the strictly dissipative case}  
\label{sec_strict}
The lower bound part of our approach presented in the previous section 
is available also for the {\em strictly dissipative} case, that is 
the case where \eqref{weak_gi} and (${\bf A}_+$) are satisfied. 
In fact, we can show the following.
\begin{thm} \label{thm_lower_str}
Assume that \eqref{weak_gi} and (${\bf A}_+$) are satisfied. 
If $\psi$ does not identically vanish, we can choose $\eps_0>0$ 
such that the global solution $u$ to \eqref{eq}--\eqref{data_eps} 
satisfies
\[
\liminf_{t\to +\infty} (\log t)^{1/2}\|u(t)\|_{L^2}>0
\]
for $\eps \in (0,\eps_0]$.
\end{thm}

\begin{rmk}
For the power-type nonlinearity case \eqref{nls_1}, similar result 
has been obtained by Kita-Sato~\cite{KitaSato}. In other words, 
Theorem~\ref{thm_lower_str} is an extension of their result to derivative 
nonlinearity case.
\end{rmk}

\begin{proof}
We set $c_*=-\sup_{\xi\in \R} \imagpart \nu(\xi)$. 
Suppose that $\hat{\psi}(\xi_*)\ne 0$ at some $\xi_*\in\R$. 
Then, in exactly the same way as the proof of Theorem~\ref{thm_lower}, 
we have 
\begin{align}
\label{est_lower_u_str}
\|u(t)\|_{L^2_x}
\ge
\left(\int_{\R} 
\frac{|A(0,\xi)|^2}{ 1+2c_* |A(0,\xi)|^2 \log t}  \, d\xi\right)^{1/2}
-\frac{C\eps^2}{t^{\kappa-\delta}}
\end{align}
if $\eps$ is suitably small. 
Also we can take an open interval $I=(\xi_*-r,\xi_*+r)$ 
with small $r>0$ such that 
\begin{align*}
 \inf_{\xi \in I}|A(0,\xi)|
 &\ge 
 \frac{\eps|\hat{\psi}(\xi_*)|}{4}>0.
\end{align*}
With $M=\sup_{\xi\in I} |A(0,\xi)|$, we obtain
\begin{align}
\int_{\R} \frac{|A(0,\xi)|^2}{ 1+2c_* |A(0,\xi)|^2 \log t}  \, d\xi
\ge
\int_{I} 
\frac{\eps^2|\hat{\psi}(\xi_*)|^2}{16( 1+2c_* M^2 \log t)}  \, d\xi
=
\frac{r \eps^2|\hat{\psi}(\xi_*)|^2}{8( 1+2c_* M^2 \log t)}.
\label{est_lower_A0_str}
\end{align}
By \eqref{est_lower_u_str} and \eqref{est_lower_A0_str}, we arrive at 
the desired lower estimate.
\end{proof}

\begin{rmk} \label{rmk_upper_str}
It would be natural to ask if the sharp upper bound could be also obtained 
by the similar approach. However, it is not trivial at all to specify the 
decay rate as $t\to +\infty$ of the integral 
\[
\int_{\R} \frac{|A(0,\xi)|^2}{ 1-2\imagpart \nu(\xi)|A(0,\xi)|^2 \log t}\, d\xi
\]
in general, because it depends essentially on behavior of $|A(0,\xi)|$ as 
$|\xi|\to \infty$. Going back to the $x$-side, this corresponds to regularity 
of the initial data 
(see e.g., \cite{Sat1} and the references cited therein for the details 
on this issue). 
This should be contrasted with the weakly dissipative case, and 
the authors have no idea so far how to handle the strictly 
dissipative case generally.
\end{rmk}

\appendix\def\thesection{Appendix}
\section{Some backgrounds on $\nu(\xi)$}  
\label{sec_app_a}
We shall give a few comments on  the backgrounds on $\nu(\xi)$. 
Without loss of generality, the cubic nonlinear term $N$ satisfying 
\eqref{weak_gi} can be written explicitly as
\begin{align*}
 N(u,u_x)=
 &a_1 u^2 u_x +a_2 u u_x^2 +a_3 u_x^3\\
 &+ b_1 \cc{u^2 u_x} + b_2 \cc{u u_x^2} 
 + b_3 \cc{u_x^3}
 \\
 &+
 c_1 \cc{u^2} u_x +  c_2  |u|^2 \cc{u_x} 
 +
 c_3 u \cc{u_x^2} + c_4 |u_x|^2 \cc{u} + c_5 |u_x|^2 \cc{u_x}\\
&+
\lambda_1 |u|^2 u + \lambda_2 |u|^2 u_x +\lambda_3 u^2\cc{u_x}
+
\lambda_4 |u_x|^2 u + \lambda_5 \cc{u} u_x^2 + \lambda_6 |u_x|^2 u_x 
\end{align*}
with suitable coefficients $a_j$, $b_j$, $c_j$, $\lambda_j \in \C$. 
With this expression of $N$, we have
\begin{align*}
N(e^{i\theta},i\xi e^{i\theta})
=&
(ia_1\xi -a_2\xi^2 -ia_3\xi^3)e^{i3\theta}\\
&+
(-ib_1\xi -b_2\xi^2 +ib_3\xi^3)e^{-i3\theta}\\
&+
(ic_1\xi-ic_2\xi -c_3\xi^2+c_4\xi^2-ic_5\xi^3)e^{-i\theta}\\
&+
(\lambda_1 +i\lambda_2\xi -i\lambda_3\xi +\lambda_4\xi^2
-\lambda_5\xi^2+i\lambda_6 \xi^3)e^{i\theta},
\end{align*}
whence 
\begin{align*}
\nu(\xi)
=&
\frac{1}{2\pi}\int_0^{2\pi} N(e^{i\theta},i\xi e^{i\theta})
e^{-i\theta}\, d\theta\\
=&
\lambda_1 +i(\lambda_2 -\lambda_3)\xi +(\lambda_4-\lambda_5)\xi^2
+i\lambda_6 \xi^3.
\end{align*}
We see that $\nu(\xi)$ depends only on the coefficients of 
the gauge-invariant terms. In this sense, $\nu(\xi)$ extracts 
the contribution from the gauge-invariant part of $N$, and 
in particular, its imaginary part is expected to have something 
to do with the dissipativity of the nonlinear terms. 
As pointed out in \cite{Su2}, \cite{Su3}, \cite{SagSu}, \cite{LiS}, 
\cite{LNSS3}, etc., 
it is the case at the level of the profile equation. 

\medskip
\subsection*{Acknowledgments}
The authors are grateful to Dr.Takuya Sato for useful conversations on 
his recent works \cite{Sat1}, \cite{Sat2}, \cite{KitaSato}, which motivate the
present work. 
Thanks are also due to Professors Satoshi Masaki, Soichiro Katayama 
and Mamoru Okamoto for their helpful comments on the authors' preceding work 
\cite{LNSS3}. 

The work of C.~L. is supported by Education 
Department of Jilin Province of China (No.~JJKH20220527KJ). 
The work of Y.~N. is supported by JSPS Grant-in-Aid for JSPS Fellows. 
The work of H.~S. is supported by JSPS KAKENHI (21K03314). 
This work is partly supported by 
Osaka Central Advanced Mathematical Institute, Osaka Metropolitan University 
(MEXT Joint Usage/Research Center on Mathematics and Theoretical 
Physics JPMXP0619217849). 


\end{document}